\documentclass[12pt]{amsart}
\usepackage{graphicx}
\usepackage[headings]{fullpage}
\usepackage{amssymb,epic,eepic,epsfig,amsbsy,amsmath,amscd,color}
\numberwithin{equation}{section}
                        \theoremstyle{plain}
\usepackage{mathrsfs}

\newcommand\no[1]{}

\newtheorem{theorem}{Theorem}[section]
\newtheorem{thm}{Theorem}
\newtheorem{lemma}[theorem]{Lemma}
\newtheorem{corollary}[theorem]{Corollary}
\newtheorem{proposition}[theorem]{Proposition}

\theoremstyle{definition}

\newcommand{\G}{G(K)}
\newcommand{\D}{\Delta}

\def\BA{\mathbb A}
\def\BC{\mathbb C}
\def\C{\mathbb C}
\def\BH{\mathbb H}
\def\BN{\mathbb N}
\def\BZ{\mathbb Z}

\def\CK{\mathcal K}

\def\CT{\mathcal T}

\def\la{\langle}
\def\ra{\rangle}

\DeclareMathOperator{\tr}{\mathrm tr}

\def\be { \begin{equation} }
\def\ee { \end{equation} }

\begin{document}
\allowdisplaybreaks
\baselineskip16pt
\title[Fibering of DTK via the adjoint hyperbolic torsion polynomial]{Fibering of double twist knots via the adjoint hyperbolic torsion polynomial}

\author[Anh T. Tran]{Anh T. Tran}



\address{Department of Mathematical Sciences,
The University of Texas at Dallas,
Richardson, TX 75080-3021, USA}
\email{att140830@utdallas.edu}

\begin{abstract}
For a hyperbolic knot $K$ in $S^3$, the adjoint hyperbolic torsion polynomial $\mathcal T^{\mathrm{Ad}}_K(t) \in \mathbb C[t^{\pm 1}]$ is defined as a normalization of the twisted Alexander polynomial of $K$ associated with the $\mathrm{SL}_3(\mathbb C)$-representation obtained by composing the holonomy representation of $K$ with the adjoint action of $\mathrm{SL}_2(\mathbb C)$ on its Lie algebra $\mathfrak{sl}_2(\mathbb C)$. In this paper we consider the adjoint hyperbolic torsion polynomial for a two-parameter family of rational knots called double twist knots, and show that $\mathcal T^{\mathrm{Ad}}_K(t)$ determines the genus and fibering of this family by using algebraic integers.
\end{abstract}

\maketitle

\section{Introduction}

The twisted Alexander polynomial was first defined by Lin \cite{Li} for knots in $S^3$ and then by Wada \cite{Wa} for finitely presentable groups. It  is a generalization of the classical Alexander polynomial by using linear representations, and has been very useful in low dimensional topology and knot theory.  Let $K \subset S^3$ be a knot and $G(K)$ its knot group, which is the fundamental group of the knot exterior $S^3 \setminus K$. Following \cite{Wa}, for a $d$-dimensional linear representation $\rho: G(K) \to \mathrm{SL}_d(\BC)$ we can define a rational function $\Delta^{\rho}_{K}(t) \in \BC(t)$ up to multiplication by $\pm t^i \, (i \in \BZ)$. We call $\Delta^{\rho}_{K}(t)$ the twisted Alexander polynomial of $K$ associated with $\rho$. Like the classical Alexander polynomial $\Delta_K(t)$, the twisted Alexander polynomial $\Delta^{\rho}_{K}(t)$ contains information about the genus and fibering of $K$. For example,  the degree of $\Delta^{\rho}_{K}(t)$ is bounded above by $d(2g(K)-1)$, where $g(K)$ is  the least genus of all Seifert surfaces bounding $K$ \cite{FK}. In this paper we say that $\Delta^{\rho}_{K}(t)$ determines the knot genus $g(K)$ if its degree is exactly equal to $d(2g(K)-1)$. Regarding fibering, it is known that if $K \subset S^3$ is a fibered knot then $\Delta^{\rho}_{K}(t)$ is expressed as a rational function of monic polynomials in $t$ \cite{GKM}. Here a Laurent polynomial in $t$ is said to be monic if its leading coefficient equals $\pm 1$. 

We now focus on  two-dimensional representations $\rho: G(K) \to \mathrm{SL}_2(\BC)$. Then the  twisted Alexander polynomial  $\Delta^{\rho}_{K}(t)$ becomes a Laurent polynomial in $t$ if $\rho$ is a nonabelian $\mathrm{SL}_2(\C)$-representation \cite{KM}. When $K \subset S^3$ is a hyperbolic knot (namely, the knot exterior admits a complete hyperbolic metric of finite volume), up to conjugation there is a unique discrete and faithful representation $\bar{\rho}_{0}: G(K) \to \mathrm{Isom}^+(\BH^3) \cong \mathrm{PSL}_2(\BC)$. This is called the holonomy representation corresponding to the hyperbolic structure.  It lifts to a representation $\rho_{0}: G(K) \to \mathrm{SL}_2(\BC)$ which is also discrete and faithful \cite{Th}. By abuse of terminology, we also call $\rho_{0}$ the holonomy representation of $K$.  In \cite{DFJ},  Dunfield, Friedl and Jackson studied the twisted Alexander polynomial $\D_{K}^{\rho_{0}}(t) \in \BC[t^{\pm 1}]$ for a hyperbolic knot $K$. It is normalized as a symmetric Laurent polynomial in $t$, which is denoted by $\CT_K(t)$ and called the \textit{hyperbolic torsion polynomial} of $K$. 
Based on extensive experiments with all hyperbolic
knots having at most $15$ crossings, they conjectured that $\CT_K(t) \in \BC[t^{\pm 1}]$ determines the knot genus $g(K)$, in the sense that $\deg \CT_K(t)= 2(2g(K)-1)$. Moreover, $K$ is fibered if and only if $\CT_K(t)$ is monic.  This conjecture has been verified for some infinite families of hyperbolic knots \cite{Mo-twist, MT-dtk, AD, Po, Mo-pretzel, MT-pretzel}. 

There is another way to get a torsion polynomial by considering the twisted Alexander polynomial  associated with the $\mathrm{SL}_3(\BC)$-representation obtained by composing the holonomy representation $\rho_{0}: G(K) \to \mathrm{SL}_2(\BC)$ with the adjoint action, denoted by $\mathrm{Ad}$, of $\mathrm{SL}_2(\BC)$ on its Lie algebra $\mathfrak{sl}_2(\BC)$. In \cite{DY},  Dubois and Yamaguchi proved that the twisted Alexander polynomial $\D_{K}^{\mathrm{Ad} \circ \rho_{0}}(t)$ is actually a Laurent polynomial in $t$. Moreover, it can be normalized as a Laurent  polynomial $\CT^{\mathrm{Ad}}_K(t)$ such that $\CT^{\mathrm{Ad}}_K(t^{-1}) = - \CT^{\mathrm{Ad}}_K(t)$ up to multiplication by $t^i \, (i \in \BZ)$. We call $\CT^{\mathrm{Ad}}_K(t)$ the \textit{adjoint hyperbolic torsion polynomial} of $K$. Like the hyperbolic torsion polynomial $\CT_K(t)$, we can ask whether $\CT^{\mathrm{Ad}}_K(t)$ determines the genus and fibering of $K$ or not. Note that $\deg \CT^{\mathrm{Ad}}_K(t) \le 3(2g(K)-1)$. In \cite{DFJ}, Dunfield, Friedl and Jackson pointed out that $\CT^{\mathrm{Ad}}_K(t)$ does not always determine the knot genus $g(K)$, in the sense that there exist  hyperbolic knots $\CK \subset S^3$ such that $\deg \CT_{\CK}(t) < 3(2g(\CK)-1)$. However, for all hyperbolic knots with at most $15$ crossings, they numerically checked that $K$ is fibered if and only if $\CT^{\mathrm{Ad}}_K(t)$ is monic. 

In this paper we consider the adjoint hyperbolic torsion polynomial $\CT^{\mathrm{Ad}}_K(t)$ for a two-parameter family of rational knots called double twist knots. Let $J(k,l)$ denote the double twist knot/link indicated in Figure \ref{pic}, where the integers $k$ and $l$ determine the signed number of half twists in the boxes. In the $k$-box (resp. $l$-box), negative (resp. positive) numbers correspond to right-handed half twists and positive (resp. negative) numbers correspond to left-handed half twists. Note that $J(k,l)$ is the same as the rational knot/link corresponding to the continued fraction $-k+\frac{1}{l}$. It is a knot when $kl$ is even (and a two-component link if $kl$ is odd) and is the trivial knot if $kl = 0$. Moreover, $J(k,l)$ is
ambient isotopic to $J(l,k)$ and $J(-k,-l)$ is the mirror image of $J(k, l)$. 
Therefore, up to mirror image, we will only consider double twist knots $J(k, 2n)$ for $k>0$ and $n \not=0$. The knot $J(k, 2n)$ is 
hyperbolic unless it is $J(2, -2)$ (the trefoil knot) or $J(1, 2n)$.  
Note that the knots $J(2m, 2n)$ are also known as genus one two-bridge knots and $J(2,2n)$ as twist knots. We will say that $J(k, 2n)$ is an even (resp. odd) double twist knot if $k$ is even (resp. odd). 

\begin{figure}[h] 
		\centering
		\includegraphics[scale=0.45]{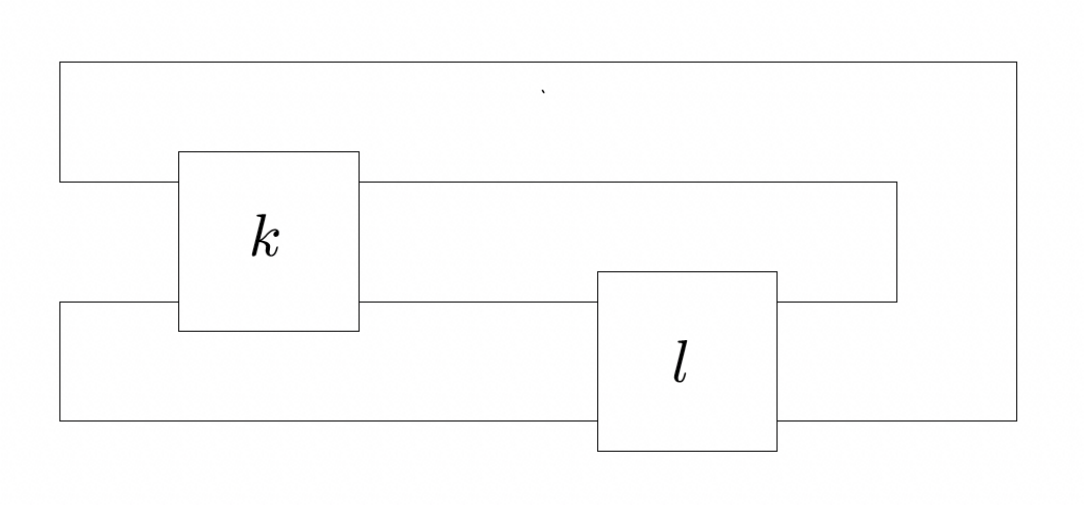}
		\caption{The double twist knot/link $J(k,l)$.}
		\label{fig:genKnotDiagram}
\label{pic}
\end{figure}

In \cite{Tr-atap}, an explicit formula for the adjoint twisted Alexander polynomial $\D_{K}^{\mathrm{Ad} \circ \rho}(t)$ of even double twist knots $J(2m, 2n)$ was provided. The fact that these knots have genus one makes it possible to perform the calculations. The formula was then applied in \cite{Mo-atap, Tr-atap2025} to show that the adjoint hyperbolic torsion polynomial $\CT^{\mathrm{Ad}}_K(t)$ determines the genus and fibering of $J(2m, 2n)$. The case of odd double twist knots $J(2m+1, 2n)$ is more challenging, since it is difficult to compute the adjoint twisted Alexander polynomial explicitly.  However, by using properties of algebraic integers, we will prove that the adjoint hyperbolic torsion polynomial  determines the genus and fibering of all double twist knots without knowing its explicit formula. 

\begin{thm} \label{main}
Let $K=J(k,2n)$ be a hyperbolic double twist knot. Then the adjoint hyperbolic torsion polynomial $\CT^{\mathrm{Ad}}_K(t) \in \BC[t^{\pm 1}]$ determines the genus of $K$, in the sense that $\deg \CT^{\mathrm{Ad}}_K(t) = 3(2g(K)-1)$. Moreover, $K$ is fibered if and only if $\CT^{\mathrm{Ad}}_K(t)$ is monic. 
\end{thm}

We remark that a similar, but incomplete, result for the hyperbolic torsion polynomial $\CT_K(t)$ of double twist knots was proved in \cite{MT-dtk}. More precisely, Morifuji and the author verified that $\CT_K(t)$ determines the genus of all double twist knots, and moreover  it determines the fibering of $J(2m+1, 2n)$. However, it is still unknown whether $\CT_K(t)$ determines the fibering of $J(2m, 2n)$ or not.  

The paper is organized as follows. In Section \ref{prelim} we review the twisted Alexander polynomial, hyperbolic torsion conjecture, and algebraic integers. In Section \ref{proof:main} we first compute the Riley polynomial of double twist knots, then determine the highest and lowest degree terms of the adjoint twisted Alexander polynomial associated with the holonomy representation, and finally give a proof of Theorem \ref{main}.

\section{Preliminaries} \label{prelim}

\subsection{Twisted Alexander polynomial}
Let $K \subset S^3$ be a knot and $E_K= S^3 \setminus K$ the knot exterior. Let $G(K)=\pi_1(E_K)$ be the knot group, which is the fundamental group of $E_K$. 
We choose  a Wirtinger presentation:
$$G(K) = \la x_1,\ldots,x_l\,|\,r_1,\ldots,r_{l-1}\ra.$$ The abelianization homomorphism 
$$
\mathfrak{a}:G(K) \to H_1(E_K;\BZ)\cong\BZ
=
\la t \ra
$$
is given by assigning to each generator $x_i$ the meridian element 
$t \in H_1(E_K;\BZ)$. Here we denote the sum in $\BZ$ multiplicatively. 

Let $\rho:\G\to \mathrm{SL}_d(\C)$ be a $d$-dimensional linear representation. The maps $\rho$ and $\mathfrak{a}$ naturally induce two ring
homomorphisms $\tilde{\rho}: {\BZ}[\G] \rightarrow M_d(\C)$ and
$\tilde{\mathfrak{a}}:{\BZ}[\G]\rightarrow {\BZ}[t^{\pm1}]$, where
${\BZ}[\G]$ is the group ring of $\G$ and $M_d(\C)$ is the
matrix algebra of degree $d$ over ${\C}$. The tensor product 
$\tilde{\rho}\otimes\tilde{\mathfrak{a}}$ defines a ring homomorphism
${\BZ}[\G] \to M_d \left({\C}[t^{\pm1}]\right)$. 
Let $F_l$ denote the
free group on the generators $x_1,\ldots,x_l$ and
$\Phi:{\BZ}[F_l]\to M_d\left({\C}[t^{\pm1}]\right)$ 
the composition of the surjection
$\tilde{\phi}:{\BZ}[F_l]\to{\BZ}[\G]$
induced by the presentation of $\G$
and the map
$\tilde{\rho}\otimes\tilde{\mathfrak{a}}:{\BZ}[\G]\to M_d\left({\C}[t^{\pm1}]\right)$.

Let $Q$ denote the $(l-1)\times l$ matrix 
whose $(i,j)$-entry is the $d\times d$ matrix
$\Phi\left(\frac{\partial r_i}{\partial x_j}\right)
\in M_d\left({\C}[t^{\pm1}]\right)$, 
where
$\frac{\partial}{\partial x}$
denotes the Fox derivative. 
For $1\leq j\leq l$,
we denote by $Q_j$
the $(l-1)\times(l-1)$ matrix obtained from $Q$
by removing the $j$-th column.
We regard $Q_j$ as
a $d(l-1)\times d(l-1)$ matrix with coefficients in
${\C}[t^{\pm1}]$. 
Then Wada's twisted Alexander polynomial \cite{Wa} of $K$
associated with a representation $\rho: \G \to \mathrm{SL}_d(\C)$ is defined
to be the rational function
$$
\D_{K}^{\rho}(t)
=\frac{\det Q_j}{\det\Phi(1-x_j)}.
$$
It is well-defined up to multiplication by $\pm t^{i}~(i\in{\BZ})$. 

Now we focus on two-dimensional linear representations of knot groups. It is known that if $\rho:\G\to \mathrm{SL}_2(\C)$ is a nonabelian representation, the rational function $\D_{K}^{\rho}(t)$ becomes a Laurent polynomial in $t$ \cite[Theorem 3.1]{KM}. 

Let $\mathrm{sl}_2(\BC)$ be the Lie algebra of $2 \times 2$ complex matrices with trace zero. The adjoint action, denoted by $\mathrm{Ad}$, is the conjugation on $\mathrm{sl}_2(\BC)$ by $\mathrm{SL}_2(\BC)$, i.e. $\mathrm{Ad}(V)(g)=VgV^{-1} \in \mathrm{sl}_2(\BC)$ for all $V \in \mathrm{SL}_2(\BC)$ and $g \in \mathrm{sl}_2(\BC)$. For each representation $\rho: G(K) \to \mathrm{SL}_2(\C)$, the composition $\mathrm{Ad} \circ \rho$ is  a representation of $G(K)$ into $\mathrm{SL}_3(\BC)$ and thus we can define a rational function $\Delta_K^{\mathrm{Ad} \circ \rho} (t)$ called the adjoint twisted Alexander polynomial of $K$.

\subsection{Hyperbolic torsion conjecture} 

For a hyperbolic knot $K$ in $S^3$, recall that $\rho_{0}: G(K) \to \mathrm{SL}_2(\BC)$ is a lift of  the holonomy representation $\bar{\rho}_{0}: G(K) \to \mathrm{PSL}_2(\BC)$ corresponding to the hyperbolic structure of the knot exterior. The representation $\rho_{0}$ is known to be discrete and faithful \cite{Th}.
In \cite{DFJ},  Dunfield, Friedl and Jackson studied the hyperbolic torsion polynomial $\CT_K(t)$ which is a normalization of the twisted Alexander polynomial $\D_{K}^{\rho_{0}}(t) \in \BC[t^{\pm 1}]$ such that $\CT_K(t) = \CT_K(t^{-1})$. 
Based on many numerical computations, they conjectured that $\CT_K(t) \in \BC[t^{\pm 1}]$ determines the knot genus $g(K)$, namely, $\deg \CT_K(t)= 2(2g(K)-1)$. Moreover, $K$ is fibered if and only if $\CT_K(t)$ is monic.  This conjecture has been verified for some families of hyperbolic knots \cite{Mo-twist, MT-dtk, AD, Po, Mo-pretzel, MT-pretzel}.

On the other hand, we can also study the adjoint twisted Alexander polynomial associated with the holonomy representation. In \cite{DY}, Dubois and Yamaguchi proved that the twisted Alexander polynomial $\D_{K}^{\mathrm{Ad} \circ \rho_{0}}(t)$ is actually a Laurent polynomial in $t$. The adjoint hyperbolic torsion polynomial $\CT^{\mathrm{Ad}}_K(t) \in \BC[t^{\pm 1}]$  is a normalization of $\D_{K}^{\mathrm{Ad} \circ \rho_{0}}(t)$ such that $\CT^{\mathrm{Ad}}_K(t^{-1}) = - \CT^{\mathrm{Ad}}_K(t)$. In \cite{DFJ}, Dunfield, Friedl and Jackson remarked that $\CT^{\mathrm{Ad}}_K(t)$ does not always determine the knot genus. However, for all hyperbolic knots having at most $15$ crossings, they numerically checked that $K$ is fibered if and only if $\CT^{\mathrm{Ad}}_K(t)$ is monic. In \cite{Mo-atap, Tr-atap2025}, by using an explicit formula for the adjoint twisted Alexander polynomial $\D_{K}^{\mathrm{Ad} \circ \rho}(t)$ from \cite{Tr-atap}, Morifuji and the author proved that  the adjoint hyperbolic torsion polynomial determines the genus and fibering of genus one two-bridge knots.

\subsection{Algebraic integers}

An \textit{algebraic integer} is a complex number which is a root of some monic polynomial whose coefficients are integers. It is easily checked that the set of all algebraic integers, denoted by $\BA$, is closed under addition and multiplication. Hence $\BA$ is a commutative subring of $\BC$. We will show the following. 

\begin{proposition} \label{alg}
Suppose $P(t) = c_q t^q + \cdots + c_ 1 t + c_0 \in \BZ[t]$ is a non-constant polynomial, and there exists a prime number $p$ such that $p \nmid c_0$ and $p \mid c_1, \cdots, c_q$. Then none of the roots of $P[t]$ are algebraic integers.
\end{proposition}

\begin{proof}
The proof is similar to that of  Eisenstein's irreducibility criterion. 

Suppose $t_0$ is a root of $P(t)$. Let $R(t)$ be the minimal polynomial of $t_0$ over $\BZ$. Write $R(t) = a_r t^r + \cdots + a_1 t + a_0 \in \BZ[t]$. We claim that $p \mid a_r$. 

Note that $R \mid P$. If $\deg R = \deg P$, then $P= d R$ for some $d \in \BZ$. Since $p \nmid c_0= da_0$ we have $p \nmid d$. Then $p \mid c_q =  da_r$ implies that $p \mid a_r$. 

Suppose $\deg R < \deg P$ then $P(t) = R(t) S(t)$ for some non-constant polynomial $S(t) \in \BZ[t]$. Write $S(t) = b_s t^s + \cdots + b_1 t + b_0$. Note that $p \nmid c_0=a_0 b_0$. Let $i \le r$ be the largest integer such that $p \nmid a_i$, and $j \le s$ the largest integer such that $p \nmid b_j$. Consider 
$$c_{i+j} = \sum_{u+v=i+j} a_u b_{v} = a_i b_j + \sum_{\substack{u+v=i+j, \\ (u,v) \not= (i,j)}} a_u b_{v}.$$

If $u>i$ then $p \mid a_u$. If $u<i$ then $v=i+j-u>j$ and so $p \mid b_v$. This implies that $p \mid a_u b_v$ if $u \not= i$. Hence 
$$
p \mid \sum_{\substack{u+v=i+j, \\ (u,v) \not= (i,j)}} a_u b_{v}.
$$
Since $p \nmid a_ib_j$, we obtain $p \nmid c_{i+j}$. This only occurs when $i+j=0$, namely, $i=j=0$. Hence $p \mid a_u$ for all $1 \le u \le r$. In particular $p \mid a_r$. This implies that $a_r \not= \pm 1$ and therefore $t_0$ cannot be an algebraic ineteger.
\end{proof}

\begin{corollary} \label{algint}
Let $P(t) \in \BZ[t]$ be a non-constant polynomial. Suppose $q$ and $r$ are co-prime integers such that $|q| \ge 2$. Then none of the roots of $Q(t):=q P(t)+r$ are algebraic integers.
\end{corollary}

\begin{proof}
Let $p$ be any prime factor of $q$. Then all coefficients of $q P(t)$ are divisible by $p$. Since $r$ is co-prime with $q$, it is not divisible by $p$. Hence all coefficients, except the constant coefficient, of $Q(t) = q P(t)+r$ are divisible by $p$.
Then, by Proposition \ref{alg}, none of the roots of $Q(t)$ are algebraic integers.
\end{proof}

\section{Double twist knots} \label{proof:main}

In this section, we first compute the Riley polynomial of double twist knots, then determine the highest and lowest degree terms of the adjoint twisted Alexander polynomial associated with the holonomy representation, and finally give a proof of Theorem \ref{main}.

It is known that $\CK := J(k, 2n)$, where $k \ge 1$ and $n \not= 0$, is fibered if and only if
\begin{itemize}
\item $\CK=J(1,2n)$, 
\item $\CK=J(2, \pm 2)$ (the trefoil knot and figure eight knot, respectively),
\item $\CK=J(3,2n)$ for $n \ge 1$.
\end{itemize}
Moreover, its genus is given by
$$
g(\CK) =
\begin{cases}
1 & \text{if~} k=2m,\\
|n| & \text{if~}  k=2m+1 \ge 3, \\
n-1 & \text{if~}  k=1 \text{~and~} n \ge 1, \\
|n| & \text{if~}  k=1 \text{~and~} n \le -1.
\end{cases}
$$
This can be proved by computing  the Alexander polynomial and applying the fact that an alternating knot is  fibered if and only if its Alexander polynomial is monic, i.e., with leading
coefficient $\pm 1$. Moreover, the degree of the Alexander polynomial of an alternating knot is equal to twice the knot genus. Note that double twist knots  are two-bridge knots which are alternating. See \cite[Lemma 7.3]{MPL} and references therein. 

By \cite{HS}, the knot group is $G(\CK)= \la a, b \mid w^n a = bw^n \ra$ where 
$$
w =  \begin{cases}
(ba^{-1})^m(b^{-1}a)^m& \text{if~} k=2m,\\
(ba^{-1})^mba(b^{-1}a)^m & \text{if~}  k=2m+1.
\end{cases}
$$

\subsection{The Riley polynomial}

A representation $\rho: G(\CK) \to \mathrm{SL}_2(\BC)$ is called non-abelian if its image $\rho(G(\CK))$ is a non-abelian subgroup of $\mathrm{SL}_2(\BC)$. Taking conjugation if necessary, we can assume that $\rho$ has the form 
\begin{equation} \label{matrix}
\rho(a) = \left[ \begin{array}{cc}
M & 1 \\
0 & M^{-1} \end{array} \right]  \quad \text{and} \quad \rho(b) = \left[ \begin{array}{cc}
M & 0 \\
2-y & M^{-1} \end{array} \right],
\end{equation}
where $(M, y) \in \BC^2$ satisfies $\rho(w^n a) = \rho(bw^n)$. Note that $y = \tr \rho(ab^{-1})$. By \cite{Ri1}, the matrix equation $\rho(w^n a) = \rho(bw^n)$ is equivalent to a single polynomial equation $R_\CK(x,y)=0$, where $x = M+M^{-1}$ and $R_K(x, y):=\rho(w)_{11} + (M^{-1} - M) \rho(w)_{12} \in \BC[x, y]$ is the Riley polynomial of a two-bridge knot $K$. Here $\rho(w)_{ij}$ is the $ij$-th entry of the matrix $\rho(w)$. 

To determine the Riley polynomial of $\CK = J(k,2n)$, we first compute $\rho(w)$ via Chebyshev polynomials.  Let $S_l(v), \, l \in \BZ,$ be the Chebyshev polynomials of the second kind defined by 
$S_0(v)=1$, $S_1(v)=v$ and $S_{l}(v) = v S_{l-1}(v) - S_{l-2}(v)$ 
for all integers $l$. 

The following properties of $S_l(v)$'s are well-known, see e.g. \cite{MT-pretzel}. 

\begin{lemma} \label{chev}
$(1)$ If $v = a + a^{-1} \not= \pm 2$ then $S_l(v) = (a^{l+1} - a^{-l-1})/(a-a^{-1})$. Moreover, $S_l( \pm 2)=(\pm 1)^l (l+1)$. 

$(2)$ $S_{-l}(v) = - S_{l-2}(v)$. 

$(3)$ For $l \ge 1$, we have $S_l(v) = \prod_{j=1}^l ( v - 2\cos \frac{j\pi}{l+1})$ and $S_l(v) - S_{l-1}(v) = \prod_{j=1}^l ( v - 2\cos \frac{(2j-1)\pi}{2l+1})$. 
In particular, all the roots of $S_l(v)$ and $S_l(v) - S_{l-1}(v)$ are real numbers in the interval $(-2,2)$. 

$(4)$  $S^2_l(v) - v S_l(v) S_{l-1}(v) + S^2_{l-1}(v) = 1$. 
\end{lemma}

We will also use the following result whose proof is based on the Cayley-Hamilton theorem $V^2 = (\tr V) V  - ( \det V)I$ for $2 \times 2$ matrices $V \in M_2(\BC)$, where $I$ denotes the identity $2 \times 2$ matrix. See e.g. \cite{MT-dtk}.

\begin{lemma} \label{expand}
Let $V \in \mathrm{SL}_2(\BC)$ and $v=\tr V$. Then 
$$
V^l = S_l(v) I - S_{l-1}(v) V^{-1} = S_{l-1}(v) V - S_{l-2}(v) I
$$
for any integer $l$. 
\end{lemma}

Then the matrix $\rho(w)$ can be computed explicitly as follows.

\begin{proposition} \label{w}
$(1)$ If $k=2m$ then
\begin{eqnarray*}
\rho(w) &=& I   - S_m(y) S_{m-1}(y)  \left[ \begin{array}{cc}
y-2 & M - M^{-1} \\
(M-M^{-1})(2-y) & 2-y \end{array} \right]\\
&& + \, S^2_{m-1}(y)  \left[ \begin{array}{cc}
(y-2)(y-M^2)  & - (M^{-1} y - M - M^{-1}) \\
(y-2)(M^{-1}y - M - M^{-1}) & M^{-2}(2-y) \end{array} \right]. 
\end{eqnarray*}

$(2)$ If $k=2m+1$ then
\begin{eqnarray*}
\rho(w) &=& I - S_m(y) S_{m-1}(y)  \left[ \begin{array}{cc}
2M^2-y & M + M^{-1} \\
(M+M^{-1}) (2-y)& 2M^{-2}-y \end{array} \right] \\
&& + \, S^2_m(y) \left[ \begin{array}{cc}
M^2-1 & M \\
M (2-y) & -y + 1 + M^{-2}  \end{array} \right]  + S^2_{m-1}(y)  \left[ \begin{array}{cc}
-y+1 + M^2 & M^{-1}  \\
M^{-1}(2-y) & M^{-2}-1 \end{array} \right]. 
\end{eqnarray*}
\end{proposition}

\begin{proof}
$(1)$ If $k=2m$, then $w = (ba^{-1})^m(b^{-1}a)^m$. Since $\tr \rho(ba^{-1}) = \tr \rho(b^{-1}a) = y$, by Lemma \ref{expand} we have
\begin{eqnarray*}
\rho(w) &=& \left( S_m(y)  I - S_{m-1}(y) \rho(ab^{-1}) \right) \left( S_m(y)  I - S_{m-1}(y) \rho(a^{-1}b) \right) \\
&=& S^2_m(y)  I - S_m(y) S_{m-1}(y) \left( \rho(ab^{-1}) + \rho(a^{-1}b) \right)+ S^2_{m-1}(y)  \rho(ab^{-1}a^{-1}b).
\end{eqnarray*}
By Lemma \ref{chev} we have $S^2_m(y) - y S_m(y) S_{m-1}(y) + S^2_{m-1}(y) = 1$. Hence
$$
\rho(w) = I - S_m(y) S_{m-1}(y) \left( \rho(ab^{-1}) + \rho(a^{-1}b) - y I\right)+ S^2_{m-1}(y)  \left( \rho(ab^{-1}a^{-1}b) - I \right).
$$

$(2)$ If $k=2m+1$ then $w = (ba^{-1})^mba(b^{-1}a)^m$. Hence
\begin{eqnarray*}
\rho(w) &=& \left( S_m(y)  I - S_{m-1}(y) \rho(ab^{-1}) \right) \rho(ba) \left( S_m(y)  I - S_{m-1}(y) \rho(a^{-1}b) \right) \\ 
&=& S^2_m(y) \rho(ba) - S_m(y) S_{m-1}(y) \left( \rho(a^2) +  \rho(b^2) \right) + S^2_{m-1}(y) \rho(ab) \\
&=&  I - S_m(y) S_{m-1}(y) \left( \rho(a^2) +  \rho(b^2) - y I\right) \\
&& + \, S^2_m(y) \left( \rho(ba) - I \right) + S^2_{m-1}(y) \left( \rho(ab) - I \right).
\end{eqnarray*}
The lemma follows by using the matrix forms of $\rho(a)$ and $\rho(b)$ in \eqref{matrix}. 
\end{proof}

Let $z: = \tr \rho(w) = \rho(w)_{11} + \rho(w)_{22}$. By taking the trace of $\rho(w)$ in Proposition \ref{w} and noting that $M^2+M^{-2} = x^2-2$, we obtain
\begin{equation} \label{z}
z =  \begin{cases}
2 + (y-2)(y+2-x^2)S^2_{m-1}(y) & \text{if~} k=2m,\\
2 - (y+2-x^2) \left( S_m(y) - S_{m-1}(y) \right)^2 & \text{if~}  k=2m+1.
\end{cases}
\end{equation}

We can now  compute the Riley polynomial $R_\CK(x,y) = \rho(w^n)_{11} + (M^{-1} - M)\rho(w^n)_{12}$. By Lemma \ref{expand} we have  $\rho(w^n) = S_{n-1}(z) \rho(w) - S_{n-2}(z) I$. Hence we can write $R_\CK = t  S_{n-1}(z) - S_{n-2}(z)$, 
where $t = \rho(w)_{11} + (M^{-1} - M) \rho(w)_{12}$. 

If $k=2m$ then by Proposition \ref{w} we have
\begin{eqnarray*}
t &=& 1  - S_m(y) S_{m-1}(y) \left( y-2 - (M-M^{-1})^2 \right) \\
&& + \, S^2_{m-1}(y)  \left( (y-2)(y-M^2) + (M-M^{-1}) (M^{-1} y - M - M^{-1}) \right) \\
&=& 1 - S_m(y) S_{m-1}(y) (y- M^2-M^{-2}) + S^2_{m-1}(y) (y-1)(y- M^2-M^{-2}).
\end{eqnarray*}
If $k=2m+1$ then it follows from Proposition \ref{w} that
\begin{eqnarray*}
t &=& 1 - S_m(y) S_{m-1}(y)  \left( 
(2M^2-y )- (M - M^{-1}) (M + M^{-1})  \right) \\
&& + \, S^2_{m-1}(y)  \left(
(-y+1 + M^2) - (M - M^{-1})  M^{-1} \right) \\
&=& 1 - S_m(y) S_{m-1}(y) (M^2+M^{-2} -y) + S^2_{m-1}(y) (M^2+M^{-2} -y).
\end{eqnarray*}

Hence $R_\CK(x,y) = t S_{n-1}(z) - S_{n-2}(z)$ where  
\begin{equation} \label{Riley}
t =  \begin{cases}
1-(y+2-x^2)S_{m-1}(y)(S_m(y) - (y-1)S_{m-1}(y)) & \text{if~} k=2m,\\
1+(y+2-x^2)S_{m-1}(y)(S_m(y) - S_{m-1}(y)) & \text{if~}  k=2m+1.
\end{cases}
\end{equation}
See  \cite{MPL, MT-dtk} for other ways to obtain \eqref{z} and \eqref{Riley}. 

\begin{proposition} \label{genus}

For $\CK = J(k, 2n)$, the following holds.

$(1)$ If $k=2m$, then $S_m(y)S_{n-1}(z) \not= 0$ for any non-abelian $\mathrm{SL}_2(\BC)$-representation.

$(2)$ If $k=2m+1 \ge 3$, then $S_m(y)S_{m-1}(y) \not= 0$ for the holonomy representation $\rho_0$. 
\end{proposition}

\begin{proof}
$(1)$ Suppose $k=2m$. If $S_{m-1}(y)=0$, then by \eqref{z} and \eqref{Riley}  we have $z=2$ and $t=1$. Hence $R_\CK = t S_{n-1}(z) - S_{n-2}(z) = S_{n-1}(2) - S_{n-2}(2) = n-(n-1) =1 \not= 0$. 

If $S_{n-1}(z)=0$, then $S^2_{n-1}(z) - z S_{n-1}(z) S_{n-2}(z) + S^2_{n-2}(z)=1$ (by Lemma \ref{chev}) implies that $S^2_{n-2}(z) = 1$. Hence $R_\CK = t S_{n-1}(z) - S_{n-2}(z) = - S_{n-2}(z) = \pm 1 \not= 0$.

Since $R_\CK=0$ for any non-abelian $\mathrm{SL}_2(\BC)$-representation, we obtain $S_m(y)S_{n-1}(z) \not= 0$. 

$(2)$ Suppose $k=2m+1$. For the holonomy representation $\rho_0$ we have $x^2=4$. 

If $S_{m-1}(y)=0$, then by Lemma \ref{chev}, $y \in (-2, 2)$. Moreover, $S^2_{m}(y) - y S_{m}(y) S_{m-1}(y) + S^2_{m-1}(y)=1$ implies that $S_m(y)= \pm 1$. By \eqref{z}  and \eqref{Riley}, we have $z = 2- (y-2) = 4-y \in (2, \infty)$ and $t=1$. Hence $R_\CK = t S_{n-1}(z) - S_{n-2}(z)  =S_{n-1}(z)-S_{n-2} (z)$. 

If $n=1$ then $R_\CK=1$. If $n \ge 2$, then since $z \in (2, \infty)$ Lemma \ref{chev}(3) implies that $R_\CK=S_{n-1}(z)-S_{n-2} (z) \not= 0$. Similarly, if $n \le -1$, then since $S_{-k}(z) = - S_{k-2}(z)$ we have $R_\CK=-S_{|n|-1}(z) + S_{|n|} (z)  \not= 0$. 

If $S_m(y)=0$, then by Lemma \ref{chev}, $y \in (-2, 2)$. Moreover, $S^2_{m}(y) - y S_{m}(y) S_{m-1}(y) + S^2_{m-1}(y)=1$ implies that $S_{m-1}(y)= \pm 1$. By \eqref{z} and \eqref{Riley}, we have $z=2- (y-2)= 4-y$ and $t=1-(y-2)=3-y$.  Hence $t = z-1$ and 
$$
R_\CK= (z-1)S_{n-1}(z)-S_{n-2} (z)=S_n(z) - S_{n-1}(z).
$$  
As in the previous case, since $z \in (2, \infty)$ we have $R_\CK \not=0$. This contradicts the fact that $R_\CK=0$ for any non-abelian $\mathrm{SL}_2(\BC)$-representation. Thus, $S_m(y)S_{m-1}(y) \not= 0$ for the holonomy representation $\rho_0$. 
\end{proof}

\subsection{The twisted Alexander polynomial}

Recall that the knot group of $\CK = J(k, 2n)$ is  $G(\CK)= \la a, b \mid w^n a = bw^n \ra$ where 
$$
w =  \begin{cases}
(ba^{-1})^m(b^{-1}a)^m& \text{if~} k=2m,\\
(ba^{-1})^mb(ab^{-1})^m a & \text{if~}  k=2m+1. 
\end{cases}
$$ . Let $r = w^naw^{-n}b^{-1}$. Then the adjoint twisted Alexander polynomial is given by $\Delta^{\mathrm{Ad} \circ \rho}_\CK(t) = \det \Phi \left( \frac{\partial r}{\partial a} \right) \big/\det\Phi(1-b)$. 

For $l \ge 0$ we let $\sigma_l(u) = \sum_{i=0}^l u^i$. We have
\begin{eqnarray} 
\frac{\partial r}{\partial a} &=& \frac{\partial w^n}{\partial a} + w^n + w^n a \frac{\partial w^{-n}}{\partial a} \nonumber\\
&=&  \begin{cases}
w^n \left( 1+ (1-a) w^{-n} \sigma_{n-1}(w) \frac{\partial w}{\partial a}\right)  & \text{if~} n \ge 1,\\
w^{n} \left( 1 - (1-a) \sigma_{|n|-1}(w) \frac{\partial w}{\partial a}\right) & \text{if~}  n \le -1. \label{dr}
\end{cases}
\end{eqnarray}

We now compute $\frac{\partial w}{\partial a}$.   
If $k=2m$ then $w =(ba^{-1})^m(b^{-1}a)^m$ and 
\begin{eqnarray*}
\frac{\partial w}{\partial a} &=& \frac{\partial (ba^{-1})^m}{\partial a} + (ba^{-1})^m \frac{\partial (b^{-1}a)^m}{\partial a} \\
&=& -ba^{-1} \sigma_{m-1}(ba^{-1}) + (ba^{-1})^m b^{-1} \sigma_{m-1}(ab^{-1}).
\end{eqnarray*}
If $k=2m+1 \ge 3$ then $w = (ba^{-1})^mb(ab^{-1})^m a$ and 
\begin{eqnarray*}
\frac{\partial w}{\partial a} &=& \frac{\partial (ba^{-1})^m}{\partial a}  + (ba^{-1})^m b  \frac{\partial (ab^{-1})^m}{\partial a} + (ba^{-1})^mb(ab^{-1})^m \\
&=& -ba^{-1} \sigma_{m-1}(ba^{-1}) + (ba^{-1})^m b \,\sigma_{m-1}(ab^{-1}) + (ba^{-1})^mb(ab^{-1})^m \\
&=& -ba^{-1} \sigma_{m-1}(ba^{-1}) + (ba^{-1})^m b \,\sigma_{m}(ab^{-1}).
\end{eqnarray*}
If $k=1$, then $w=ba$ and $\frac{\partial w}{\partial a} = b$. Hence
\begin{equation} \label{dw}
\frac{\partial w}{\partial a} =  \begin{cases}
-ba^{-1} \sigma_{m-1}(ba^{-1}) + (ba^{-1})^m b^{-1} \sigma_{m-1}(ab^{-1}) & \text{if~} k=2m,\\
-ba^{-1} \sigma_{m-1}(ba^{-1}) + (ba^{-1})^m b \, \sigma_{m}(ab^{-1}) & \text{if~}  k=2m+1 \ge 3, \\
b & \text{if~} k=1.
\end{cases}
\end{equation} 

Recall that the adjoint action $\mathrm{Ad}$ is the conjugation on $\mathrm{sl}_2(\BC)$ by $\mathrm{SL}_2(\BC)$, i.e. $\mathrm{Ad}(V)(g)=VgV^{-1} \in \mathrm{sl}_2(\BC)$ for all $V \in \mathrm{SL}_2(\BC)$ and $g \in \mathrm{sl}_2(\BC)$. We  need the following lemma.  

\begin{lemma} \cite[Lemma 2.1]{NT} \label{adjoint}
Suppose $V = \left[ \begin{array}{cc}
e& f \\
g & h \end{array} \right] \in \mathrm{SL}_2(\BC)$. Then 
$$
\mathrm{Ad}(V)=\left[ \begin{array}{ccc}
e^2 & -2ef & -f^2 \\
-eg & eh + fg & fh \\
-g^2 & 2gh & h^2 \end{array} \right].
$$
Moreover, $\mathrm{Ad}(V) \in \mathrm{SL}_3(\BC)$. 
\end{lemma}

\begin{proposition} \label{ad}
For $V \in \mathrm{SL}_2(\BC)$ and $l \in \BN$ we have 
$$\det \sigma_l \left( \mathrm{Ad}(V) \right) = (l+1) S^2_l(v)$$ where $v = \tr V$.
\end{proposition}

\begin{proof}
Up to conjugation we can assume that  $V= \left[ \begin{array}{cc}
s& c \\
0 & s^{-1} \end{array} \right]$ for some $s \not= 0$. Since $v = \tr V = s + s^{-1}$, we have $S_l(v) = (s^{l+1} - s^{-l-1})/(s-s^{-1}) = s^{-l}  (1 + s^2 + \cdots + s^{2l})$. 

By Lemma \ref{adjoint} we have $\mathrm{Ad}(V) =\left[ \begin{array}{ccc}
s^2 & -2cs & -c^2 \\
0 & 1& cs^{-1}\\
0 & 0 & s^{-2} \end{array} \right]$. Then
$$
\sigma_l (\mathrm{Ad}(V) ) = \sum_{i=0}^l (\mathrm{Ad}(V) )^i = \left[ \begin{array}{ccc}
1 + s^2 + \cdots + s^{2l} &* & * \\
0 & l+1& *\\
0 & 0 & 1 + s^{-2} + \cdots + s^{-2l} \end{array} \right].
$$
Hence
$$
\det \sigma_l (\mathrm{Ad}(V) )  = (l+1) s^{-2l}  (1 + s^2 + \cdots + s^{2l})^2= (l+1)S^2_l(v).
$$
\end{proof}

We are ready to determine the highest and lowest degree terms of the adjoint twisted Alexander polynomial $\Delta^{\mathrm{Ad} \circ \rho_0}_\CK(t)$, which is equal to the adjoint hyperbolic torsion polynomial $\CT^{\mathrm{Ad}}_\CK(t)$ up to multiplication by $\pm t^{i}~(i\in{\BZ})$. Recall from \eqref{matrix} that, up to conjugation, a non-abelian representation $\rho: G(\CK) \to \mathrm{SL}_2(\BC)$ has the form
$$
\rho(a) = \left[ \begin{array}{cc}
M & 1 \\
0 & M^{-1} \end{array} \right]  \quad \text{and} \quad \rho(b) = \left[ \begin{array}{cc}
M & 0 \\
2-y & M^{-1} \end{array} \right].
$$

Let $A=(\mathrm{Ad} \circ \rho) (a)$, $B=(\mathrm{Ad} \circ \rho) (b)$ and $W=(\mathrm{Ad} \circ \rho) (w)$, which are matrices in $\mathrm{SL}_3(\BC)$. Then $\Phi(a) = t A$, $\Phi(b) = t B$ and $\Phi(w) = \begin{cases}
t^0 W & \text{if~} k=2m,\\
t^6 W & \text{if~}  k=2m+1.
\end{cases}$

Since $\rho(b) = \left[ \begin{array}{cc}
M & 0 \\
2-y & M^{-1} \end{array} \right]$, by Lemma \ref{adjoint} we have 
$$
B = \mathrm{Ad}(\rho(b)) = \left[ \begin{array}{ccc}
M^2 & 0 & 0 \\
-(2-y)M & 1 & 0 \\
-(2-y)^2 & 2(2-y)M^{-1} & M^{-2} \end{array} \right].
$$ 
Then $\det\Phi(1-b) = (1-t)(1 - tM^2)(1 - tM^{-2}) = (1-t)(t^2-t(x^2-2)+1)$. Hence 
$$
\Delta^{\mathrm{Ad} \circ \rho}_\CK(t) = \frac{\det \Phi \left( \frac{\partial r}{\partial a} \right) }{(t-1)(t^2-t(x^2-2)+1)}.
$$

For the holonomy representation $\rho_0$, by \cite{DY} the rational function $\Delta^{\mathrm{Ad} \circ \rho_0}_\CK(t)$ becomes a Laurent polynomial in $t$. Its highest and lowest degree terms are given as follows. 

\begin{proposition} \label{pos}
 If $n \ge 1$ then
\begin{eqnarray*}
\Delta^{\mathrm{Ad} \circ \rho_0}_\CK(t) =
\begin{cases}
mn S^2_{m-1}(y) S^2_{n-1}(z) t^0 + \dots + mn S^2_{m-1}(y) S^2_{n-1}(z) t^{-3} & \text{if~} k=2m,\\
m S^2_m(y) t^{6n-3}  + \dots + m S^2_m(y) t^{0}  & \text{if~}  k=2m+1 \ge 3, \\
t^{6n-6}  + \dots + t^{3} & \text{if~} k=1 \text{~and~} n \ge 2.
\end{cases}
\end{eqnarray*}
\end{proposition}

\begin{proof}
Since $n \ge 1$, by \eqref{dr} we have $\frac{\partial r}{\partial a} =  
w^n \left( 1+ (1-a) w^{-n} \sigma_{n-1}(w) \frac{\partial w}{\partial a}\right)$. 

 If $k=2m$, then it follows from \eqref{dw} and Proposition \ref{ad} that
\begin{eqnarray*}
\det \Phi \left( \frac{\partial r}{\partial a} \right) &=& t^3 \det \left( A W^{-n} \sigma_{n-1}(W) BA^{-1} \sigma_{m-1}(BA^{-1})  \right) \\
&& + \dots + t^{-3} \det \left( - W^{-n} \sigma_{n-1}(W) (BA^{-1})^m B^{-1} \sigma_{m-1}(BA^{-1})  \right) \\
&=& t^3 n S^2_{n-1}(z) m S^2_{m-1}(y) + \dots + t^{-3} n S^2_{n-1}(z) m S^2_{m-1}(y). 
\end{eqnarray*}
Hence $
\Delta^{\mathrm{Ad} \circ \rho_0}_\CK(t) = mn S^2_{m-1}(y) S^2_{n-1}(z) t^0 + \dots + mn S^2_{m-1}(y) S^2_{n-1}(z) t^{-3}$.

If $k=2m+1 \ge 3$, then by \eqref{dw} we have
\begin{eqnarray*}
\det \Phi \left( \frac{\partial r}{\partial a} \right) &=& t^{6n} \det \left( I - AW^{-n} W^{n-1} (BA^{-1})^m B \, \sigma_{m}(AB^{-1})  \right) \\
&& + \dots + t^{0} \det \left( - W^{-n} BA^{-1} \sigma_{m-1}(BA^{-1})   \right). 
\end{eqnarray*}
Since 
\begin{eqnarray*}
I - AW^{-n} W^{n-1} (BA^{-1})^m B  \sigma_{m}(AB^{-1}) &=& I - (AB^{-1})^{-m} \sigma_{m}(AB^{-1}) \\
&=& I - \sigma_{m}(BA^{-1}) \\
&=& - BA^{-1} \sigma_{m-1}(BA^{-1}),
\end{eqnarray*}
by Proposition \ref{ad} we obtain $\det \Phi \left( \frac{\partial r}{\partial a} \right) = t^{6n}  m S^2_m(y) + \dots + t^{0} m S^2_{m-1}(y)$. Hence 
$$
\Delta^{\mathrm{Ad} \circ \rho_0}_\CK(t)   = t^{6n-3}  m S^2_m(y) + \dots + t^{0} m S^2_{m-1}(y).
$$

If $k=1$, then $w=ba$ and $\frac{\partial w}{\partial a} = b$. We have
\begin{eqnarray*}
\det \Phi \left( \frac{\partial r}{\partial a} \right) &=& t^{6n} \det \left( I + (I - t A) (t^2BA)^{-n} \sigma_{n-1}(t^2 BA) (tB) \right) \\
&=&  t^{6n} \det \left(  - (tA) (t^2BA)^{-n} (t^2BA)^{n-2} (tB) + \dots +  (t^2BA)^{-n}  (tB) \right) \\
&=& t^{6n-3} + \dots + t^3.
\end{eqnarray*}
Note that $(tA) (t^2BA)^{-n} (t^2BA)^{n-1} (tB) = I$. Hence $\Delta^{\mathrm{Ad} \circ \rho_0}_\CK(t)   = t^{6n-6}  + \dots + t^{3}$. 
\end{proof}

\begin{proposition} \label{neg}
 If $n \le -1$ then
\begin{eqnarray*}
\Delta^{\mathrm{Ad} \circ \rho_0}_\CK(t) =
\begin{cases}
m|n| S^2_{m-1}(y) S^2_{|n|-1}(z) t^0 + \dots + m|n| S^2_{m-1}(y) S^2_{|n|-1}(z) t^{-3} & \text{if~} k=2m,\\
 t^{-3} (m+1) S^2_{m}(y) + \dots + t^{6n}(m+1) S^2_{m}(y) & \text{if~}  k=2m+1.
\end{cases}
\end{eqnarray*}
\end{proposition}

\begin{proof}
Since $n \le -1$, by \eqref{dr} we have $\frac{\partial r}{\partial a} = w^{n} \left( 1 - (1-a)  \sigma_{|n|-1}(w) \frac{\partial w}{\partial a}\right)$. 

 If $k=2m$, then it follows from \eqref{dw} and Proposition \ref{ad} that
\begin{eqnarray*}
\det \Phi \left( \frac{\partial r}{\partial a} \right) &=& t^3 \det \left( - A \sigma_{|n|-1}(W) BA^{-1} \sigma_{m-1}(BA^{-1})  \right) \\
&& + \dots + t^{-3} \det \left( -  \sigma_{|n|-1}(W) (BA^{-1})^m B^{-1}\sigma_{m-1}(BA^{-1})  \right) \\
&=& t^3 |n| S^2_{|n|-1}(z) m S^2_{m-1}(y) + \dots + t^{-3} |n| S^2_{|n|-1}(z) m S^2_{m-1}(y). 
\end{eqnarray*}
Hence $
\Delta^{\mathrm{Ad} \circ \rho_0}_\CK(t) = m|n| S^2_{m-1}(y) S^2_{|n|-1}(z) t^0 + \dots + m|n| S^2_{m-1}(y) S^2_{|n|-1}(z) t^{-3}$.

If $k=2m+1 \ge 3$, then by \eqref{dw} and Proposition \ref{ad} we have
\begin{eqnarray*}
\det \Phi \left( \frac{\partial r}{\partial a} \right) &=& t^{6n} \det \left( t^{2|n|} A W^{|n|-1} (BA^{-1})^m B \, \sigma_{m}(AB^{-1})  \right) \\
&& + \dots + t^{6n} \det \left( I +  BA^{-1} \sigma_{m-1}(BA^{-1})   \right) \\
&=& t^0 (m+1) S^2_{m}(y) + \dots + t^{6n}(m+1) S^2_{m}(y). 
\end{eqnarray*}
Note that $ I +  BA^{-1} \sigma_{m-1}(BA^{-1}) = \sigma_{m}(BA^{-1})$. Hence 
$$\Delta^{\mathrm{Ad} \circ \rho_0}_\CK(t)  = t^{-3} (m+1) S^2_{m}(y) + \dots + t^{6n}(m+1) S^2_{m}(y).$$

If $k=1$, then $w=ba$ and $\frac{\partial w}{\partial a} = b$. We have
$$
\det \Phi \left( \frac{\partial r}{\partial a} \right) = t^{6n} \det \left( t^{2|n|} A W^{|n|-1} B  \right)  + \dots + t^{6n} \det  I 
= t^0  + \dots + t^{6n}. 
$$
Hence 
$\Delta^{\mathrm{Ad} \circ \rho_0}_\CK(t) = t^{-3}  + \dots + t^{6n}.$
\end{proof}

\subsection{Proof of Theorem \ref{main}} 

Recall that $\BA$ denotes the set of all algebraic integers. 

For the holonomy representation $\rho_0$ of a hyperbolic knot, we have $\tr \rho_0(\mu) = \pm 2$, where $\mu$ is any meridian. By \cite{Ri2}, the Riley polynomial $R_K(\pm 2, y) \in \BZ[y]$ of a two-bridge knot $K$ is monic. This implies that for the holonomy representation $\rho_0$ of a hyperbolic two-bridge knot, we have $y \in \BA$. 

Consider the double twist knot $\CK = J(k, 2n)$ and its holonomy representation $\rho_0$. Since $x = \pm 2$, by \eqref{z} we have
$$
z =  \begin{cases}
2 + (y-2)^2S^2_{m-1}(y) & \text{if~} k=2m,\\
2 - (y-2) \left( S_m(y) - S_{m-1}(y) \right)^2 & \text{if~}  k=2m+1.
\end{cases}
$$ 

The proof of Theorem \ref{main} is divided into three cases: $(1)$ $k=2m$, $(2)$ $k=2m+1 \ge 3$, and $(3)$ $k=1$. 

\underline{Case 1}: $k=2m$. Note that $J(2m, 2n)$ has genus one. Moreover, $J(2m, 2n)$ is fibered if and only  if $m=|n|=1$. 

Since $S_{-n-1}(z) = - S_{n-1}(z)$, by Proposition \ref{genus} we have $S^2_{m-1}(y) S^2_{|n|-1}(z) \not= 0$. Hence  Propositions \ref{pos} and \ref{neg} imply that the adjoint twisted Alexander polynomial $\Delta^{\mathrm{Ad} \circ \rho_0}_\CK(t)$ is a polynomial of degree $3$ with highest coefficient $m|n| S^2_{m-1}(y) S^2_{|n|-1}(z)$. 

If $m|n|=1$, then $m|n| S^2_{m-1}(y) S^2_{|n|-1}(z)=1$. Hence $\Delta^{\mathrm{Ad} \circ \rho_0}_\CK(t)$ is monic.

If $m|n| \ge 2$, then by applying Corollary \ref{algint} with $q=m|n|$, $r = \pm 1$ and $P(y) = S^2_{m-1}(y) S^2_{|n|-1}(z) \in \BZ[y]$ and noting that $y \in \BA$, we obtain $m|n| S^2_{m-1}(y) S^2_{|n|-1}(z) \pm 1 \not= 0$. This means that $\Delta^{\mathrm{Ad} \circ \rho_0}_\CK(t)$ is not monic. 

\underline{Case 2}: $k=2m+1 \ge 3$. Note that $J(2m+1, 2n)$ has genus $|n|$. Moreover, $J(2m+1, 2n)$ is fibered if and only if $m=1$ and $n \ge 1$. 

By Proposition \ref{genus} we have $S_m(y) \not= 0$ and $S_{m-1}(y) \not= 0$. 

If $n \ge 1$, then by Proposition \ref{pos}, $\Delta^{\mathrm{Ad} \circ \rho_0}_\CK(t)$ is a polynomial of degree $6n-3$ with highest coefficient $mS^2_{m-1}(y)$. If $m=1$, then $mS^2_{m-1}(y)=1$ and hence $\Delta^{\mathrm{Ad} \circ \rho_0}_\CK(t)$ is monic. If $m \ge 2$, then by applying Corollary \ref{algint} with $q=m$, $r = \pm 1$ and $P(y) = S^2_{m-1}(y)  \in \BZ[y]$ and noting that $y \in \BA$, we obtain  $mS^2_{m-1}(y) \pm 1 \not= 0$. Thus $\Delta^{\mathrm{Ad} \circ \rho_0}_\CK(t)$ is not monic. 

If $n \le -1$, then by Proposition \ref{neg}, $\Delta^{\mathrm{Ad} \circ \rho_0}_\CK(t)$ is a polynomial of degree $-3+6|n|$ with highest coefficient $(m+1)S^2_{m}(y)$. Since $y \in \BA$, Corollary \ref{algint} implies that $(m+1)S^2_{m}(y)  \pm 1 \not= 0$. This means that $\Delta^{\mathrm{Ad} \circ \rho_0}_\CK(t)$ is not monic.

\underline{Case 3}: $k=1$. Note that $J(1, 2n)$ is fibered and its genus is given by 
$$
g(J(1, 2n)) =
\begin{cases}
n-1 & \text{if~} n \ge 1,\\
|n| & \text{if~}  n \le -1.
\end{cases}
$$

If $n \ge 2$, then by Proposition \ref{pos}, $\Delta^{\mathrm{Ad} \circ \rho_0}_\CK(t)$ is a monic polynomial of degree $6n-9$. 

If $n \le -1$, then by Proposition \ref{neg}, $\Delta^{\mathrm{Ad} \circ \rho_0}_\CK(t)$ is a monic polynomial of degree $-3+6|n|$. 

In all  three cases shown, we have $\deg \Delta^{\mathrm{Ad} \circ \rho_0}_\CK (t) = 3 (2 g(\CK)-1)$. Moreover, $\Delta^{\mathrm{Ad} \circ \rho_0}_\CK(t)$ is monic if and only if $\CK$ is fibered. This completes the proof of Theorem \ref{main}.

\section*{Acknowledgements}

The author has been supported by a grant from the Simons Foundation (\#708778). 

\end{document}